\documentclass[11pt]{amsart}

 \usepackage{amsmath,amssymb}
 \usepackage{color}
 \usepackage{mathrsfs}
 \usepackage{graphicx}
 \usepackage{lineno}

 \usepackage{calc}
\makeatletter

\newlength{\temp@wc@width}

\newlength{\temp@wc@height}

\newcommand{\widecheck}[1]{%
\setlength{\temp@wc@width}{\widthof{$#1$}}%
\setlength{\temp@wc@height}{\heightof{$#1$}}%
#1\hspace{-\temp@wc@width}%
\raisebox{\temp@wc@height+2pt}[\heightof{$\widehat{#1}$}]%
{\rotatebox[origin=c]{180}{\vbox to 0pt{\hbox{$\widehat{\hphantom{#1}}$}}}}%
}

\makeatother
 
 \numberwithin{equation}{section}
 \theoremstyle{plain}
 \newtheorem{theorem}{Theorem}[section]

 \newtheorem{corollary}[theorem]{Corollary}
 \newtheorem{proposition}[theorem]{Proposition}

  {\theoremstyle{definition}\newtheorem{definition}[theorem]{Definition}}

 \def\bsr{\operatorname{bsr}}
\def\tsr{\operatorname{tsr}}

 \def\T{ \mathbb T}
 \def\R{ \mathbb R}
 \def\H{H^\infty}

 \def\D{{ \mathbb D}}
  \def\K{{ \mathbb K}}
 \def\C{{ \mathbb C}}
 
 \def\N{{ \mathbb N}}

 \def\bs{\boldsymbol}
 
 \def\union{\cup}
 
 \def\inter{\cap}
 
 \def\ov{\overline}

 \def\ss{\subseteq}

 \def\buildrel#1_#2^#3{\mathrel{\mathop{\kern 0pt#1}\limits_{#2}^{#3}}}
 
\begin{document}


 \title [Polydisk algebras]{Embedding polydisk algebras into the disk algebra and an application to stable ranks}


 \author{Raymond Mortini}
 \address{\small D\'{e}partement de Math\'{e}matiques\\
 Institut \'Elie Cartan de Lorraine,  UMR 7502\\
 Ile du Saulcy\\
 F-57045 Metz, France} 
 \email{Raymond.Mortini@univ-lorraine.fr}


 
 \begin{abstract}
 It is shown how to embed the polydisk algebras (finite and infinite ones) into the disk algebra $A(\ov\D)$.
 As a consequence, one  obtains uniform closed subalgebras of $A(\ov\D)$ which have arbitrarily prescribed stable ranks.
  \end{abstract}

  \maketitle

 \centerline {\small\the\day.\the \month.\the\year} \medskip
\section*{Introduction}
  Let $\D=\{z\in \C: |z|<1\}$ be  the open unit disk, $\ov\D=\{z\in \C: |z|\leq 1\}$ its closure, and
  $A(\ov\D)$  the disk-algebra, that is the space of all functions continuous on $\ov\D$ and holomorphic on $\D$.
 In this note I am interested in the question whether there are subalgebras of the disk algebra $A(\ov\D)$
that do not have the Bass stable rank one (see below for the definitions). As is well known, Jones,
Marshall and Wolff showed that the stable rank of $A(\ov\D)$ is one. Whereas in \cite{mo92}
I unveiled for any $n\in \N\union\{\infty\}$ subalgebras of $\H$ on the disk which have stable rank $n$, the problem whether these algebras could be chosen to be subalgebras of $A(\ov\D)$, remained open.   The examples
given in \cite{mo92} always meet $\H(\D)\setminus A(\ov\D)$. It is a quite  recent result developed together with Rudolf Rupp (see Corollary \ref{roy2}) that 
any subalgebra $B$  of $A(\ov\D)$   containing the polynomials
and satisfying Royden's property $(\alpha_0)$ has Bass stable rank one (note that $B$ is not  assumed to be closed in $A(\ov\D)$).
 On the other hand,  it is  easy to construct  a subalgebra of $A(\ov\D)$ that has  stable rank two:
just take  the restriction $\C[z]\,|_{\ov\D}$ of the polynomials to $\ov\D$.
In an oral communication Amol Sasane unveiled  a first non-closed subalgebra of $A(\ov\D)$
with stable rank infinity: if $\varphi$ is a conformal map of the disk $\{|z|<2\}$ onto the
upper half plane $H^+$, then the algebra 
$$A=\{f\circ\varphi|_{\ov\D}: f\in {\rm AP}^+\}$$
of pull-backs of almost periodic functions that are analytic on $H^+$ is  isomorphic to
${\rm AP}^+$ and henceforth has stable rank infinity (see \cite{misa} and  \cite{mr}).

It is the aim of this paper to prove, given $n\in \N\union \{\infty\}$,  the existence of {\it uniformly closed} subalgebras of $A(\ov\D)$
that have  Bass stable rank $n$. The proof
is based on embedding the polydisk algebras $A(\ov\D^{\,n})$ and $A({\bf D}^{\infty})$ isomorphically into $A(\ov\D)$ (see below for the definitions).
This will be done by using the Rudin-Carleson interpolation theorem for 
disk algebra functions and the  topological fact (known under the name
of the Alexandroff-Hausdorff theorem), that every compact metric space is the continuous
image of the  Cantor set (see for example \cite{ros}).

\section{background}

\begin{definition}
 Let $A$ be  a commutative unital algebra (real or complex) 
 with identity element denoted by 1.

\begin{enumerate}
\item [(1)] An $n$-tuple $(f_1,\dots,f_n)\in A^n$ is said to be {\it invertible} (or {\it unimodular}) 
if there exists
 $(x_1,\dots,x_n)\in A^n$ such that the B\'ezout equation $\sum_{j=1}^n x_jf_j=1$
 is satisfied.
   The set of all invertible $n$-tuples is denoted by $U_n(A)$. Note that $U_1(A)=A^{-1}$.
   \index{$U_n(A)$}
   
 An $(n+1)$-tuple $(f_1,\dots,f_n,g)\in U_{n+1}(A)$ is  called {\sl reducible}  
 if there exists 
 $(a_1,\dots,a_n)\in A^n$ such that $(f_1+a_1g,\dots, f_n+a_ng)\in U_n(A)$.

 \item [(2)] The {\sl Bass stable rank} of $A$, denoted by $\bsr A$,  is the smallest integer $n$ such that every element in $U_{n+1}(A)$ is reducible. 
 If no such $n$ exists, then $\bsr A=\infty$. 
  \end{enumerate}
  \end{definition}
  
 It is obvious that if  $A$ and $B$ are two commutative unital algebras such  that $A$ is isomorphic to $B$, then $\bsr A=\bsr B$, because
any isomorphism $\iota$ between $A$ and $B$ induces a bijection between $U_n(A)$ and
$U_n(B)$.
The following two observations  stem from joint work with R. Rupp \cite{moru}. 
Here $\K=\R$ or $\K=\C$.

\begin{proposition}\label{bsr-roy}
Let $X$ be a topological space and $B$ a subalgebra of  $C_b(X,\K)$ with $\K\ss B$. 
Suppose that $B$ has Royden's property $(\alpha_0)$; that is
\begin{enumerate}
\item [$(\alpha_0)$] For every $f\in B$: if $||1-f||_\infty<1$, then  $f\in B^{-1}$.
\end{enumerate}
Then $\bsr B\leq \bsr \ov B^{||\cdot||_\infty}$, where $\ov B^{||\cdot||_\infty}$ is the 
uniform closure of $B$.
\end{proposition}
\begin{proof}
Let $A:= \ov B^{||\cdot||_\infty}$.  We show that $U_n(B)=U_n(A)\inter B^n$.
Since $U_n(B)\ss U_n(A)\inter B^n$, it only remains to show the reverse inclusion. So let
$(b_1,\dots,b_n)\in U_n(A)\inter B^n$. Then there is $(a_1,\dots,a_n)\in A^n$
such that $1=\sum_{j=1}^n a_jb_j$. Uniformly approximating $a_j$ by
elements $x_j\in B$ yields that $||\sum_{j=1}^nx_jb_j-1||_\infty<1/2$. By assumption
$(\alpha_0)$, $f:=\sum_{j=1}^nx_jb_j\in B^{-1}$. Hence $(b_1,\dots,b_n)\in U_n(B)$.
It is now a  standard observation that  $\bsr B\leq \bsr A$ (see \cite{cs3} or \cite{misa}).
\end{proof}
 \begin{corollary}\label{roy2}
 Let $B$ be a subalgebra of the disk algebra $A(\ov\D)$ such that
 \begin{enumerate}
\item [(1)] $B$ contains the polynomials (that is $\C[z]\,|_{\ov \D} \ss B$);
\item [(2)] For every $f\in B$: if $||1-f||_\infty<1$, then  $f\in B^{-1}$.
\end{enumerate}
Then $\bsr B=1$.
\end{corollary}
\begin{proof}
By (1), $B$ is uniformly dense in $A(\ov\D)$. Because (2) is Royden's property
$(\alpha_0)$, we may apply Proposition  \ref{bsr-roy} to conclude that $\bsr B\leq\bsr A(\ov\D)$.
Since by the Jones-Marshall-Wolff theorem $\bsr A(\ov\D)=1$ (\cite{jmw}), we are done. 
\end{proof}

\section{An embedding theorem}

Recall that  $\ov\D^n$ is the closed polydisk  and 
${\bf D}^\infty:=\prod_{n\in \N} \ov\D$ the   infinite polydisk. 
By Tychonov's theorem, $\bf D^\infty$ is a compact metric space when endowed with
 the product topology. Moreover, each $\ov\D^n$ and ${\bf D}^\infty$ are separable.
  The {\it polydisk algebra} $A(\ov\D^n)$ is the set of functions 
 continuous on $\ov\D^n$ and holomorphic on $\D^n$.  In the same spirit, one defines
 the infinite polydisk algebra $A({\bf D}^\infty)$ as  the smallest uniformly closed subalgebra of 
 $C(\bf D^\infty,\C)$ containing all the coordinate functions $z_1,z_2,\dots$.
 Let $\C[z_1,z_2,\dots]$ denote the set of polynomials 
 $$\sum_{\bs j\in \N^n} a_{\bs j} z_1^{j_1}\dots z_n^{j_n}, n\in \N,$$
 over $\C$, where $\bs j=(j_1,\dots, j_n)\in \N^n$. Hence
 $$\C[z_1,z_2,\dots]\,|_{{\bf D}^\infty}\ss A({\bf D}^\infty).$$

 \begin{theorem}\label{embed}
 There are uniformly closed subalgebras $A_n$ and $A_\infty$  of $A(\ov\D)$ 
that are algebraically isomorphic to $A(\ov\D^n)$, respectively $A({\bf D}^\infty)$.
\end{theorem}
\begin{proof}
 Let $C\ss\T$ be the homeomorphic image of the usual ternary Cantor set on $[0,1]$ via 
 the map $e^{i\pi x}$.
By the Alexandroff-Hausdorff theorem, \cite{ros}, there is a continuous surjective map  
$$M_n=(\phi_1,\dots, \phi_n):C\to \ov\D^n,$$ 
respectively
$$M_\infty=(\phi_1,\phi_2,\dots): C\to {\bf D}^\infty.$$
Since $C$ has one-dimensional Lebesgue-measure zero, the Rudin-Carleson interpolation  Theorem \cite[p. 58]{gam}
implies that there are functions $f_j\in A(\ov\D)$ such that $f_j|_C=\phi_j$ and $||f_j||=1$. 
 Define $F_n: \ov\D\to \ov\D^n$ by
$$F_n(\xi)=(f_1(\xi),\dots, f_n(\xi)),$$
and $F_\infty: \ov\D\to {\bf D}^\infty$ by
$$F_\infty(\xi)=(f_1(\xi),f_2(\xi), \dots).$$
By construction, the range of $F_n$ on $\ov\D$ is $\ov\D^n$ and the range of $F_\infty$
on $\ov \D$ is ${\bf D}^\infty$. Moreover, since  $f_j(\D)\ss \D$, 
the functions $f\circ F_n$ and $f\circ F_\infty$ are holomorphic on $\D$ for any $f\in A(\ov\D^n)$
respectively $f\in A({\bf D}^\infty)$. Hence
$$\Psi_n: \begin{cases} A(\ov\D^n) & \to A(\ov\D)\\
 f &\mapsto f\circ F_n
 \end{cases}
 $$
 and
 $$\Psi_\infty: \begin{cases} A({\bf D}^\infty) & \to A(\ov\D)\\
 f &\mapsto f\circ F_\infty
 \end{cases}
 $$
are isometric isomorphisms of $A(\ov\D^n)$, respectively $A({\bf D}^\infty)$,  onto a uniformly closed subalgebra of $A(\ov\D)$. 
\end{proof}

 \begin{corollary}
 For every $n\in \N\union \{\infty\}$ there is a  uniformly closed subalgebra $A_n$ of $A(\ov\D)$ 
with $\bsr A_n=n$.
\end{corollary}
 \begin{proof}
 Let $N\in \N$ be chosen so that $\left\lfloor \frac{N}{2}\right\rfloor +1=n$.  By Theorem \ref{embed},
 $A(\ov\D^N)$ is isomorphic to a uniformly closed subalgebra $A_N$ of $A(\ov\D)$. Hence
 $$\bsr A_N=\bsr A(\ov\D^N)=\left\lfloor \frac{N}{2}\right\rfloor +1=n,$$ where the penultimate
  equality is due to Corach and Su\'arez \cite{cs3}.
 Moreover, by \cite{mo92}, $\bsr A({\bf D}^\infty)=\infty$.  Since by Theorem 
 \ref{embed},
 $A({\bf D}^\infty)$ is isomorphic to a uniformly closed subalgebra $A_\infty$ of $A(\ov\D)$,
 we deduce that
 $$\bsr A_\infty=\bsr A({\bf D}^\infty)=\infty.$$
\end{proof}
\section{The topological stable rank}

Associated with the Bass stable rank is the notion of {\it topological stable rank}
introduced by Rieffel \cite{ri}.
\begin{definition}
Let $A$ be  a commutative unital complex Banach algebra. The {\it topological stable rank}, 
${\rm tsr} A$, of $A$ is the least integer
  $n$ for which $U_n(A)$ is dense in $A^n$, or infinite if no such $n$ exists.  
\end{definition}
It is straightforward  to see (and well known) that $\tsr A(\ov\D)=2$. Corach and Su\'arez \cite{cs}
showed that $\tsr A(\ov\D^n)=n+1$ for $n\in \N$.  Because $\bsr A\leq \tsr A$ is always true,
  $\tsr A({\bf D}^\infty)=\infty$. Since the topological  stable rank is invariant under isometric
  isomorphisms, we  obtain from Corollary \ref{embed} the following theorem.
  
    \begin{corollary}
 For every $n\in \N\union \{\infty\}$ there is a  uniformly closed subalgebra $A_n$ of $A(\ov\D)$ 
with $\tsr A_n=n+1$.
\end{corollary}
\begin{proof}
Just take $A_n:=\Psi_n(A(\ov\D^n))$, respectively 
$A_\infty:=\Psi_\infty(A({\bf D}^\infty))$. 
\end{proof}

\subsection*{Acknowledgements}

I thank Amol Sasane and Rudolf Rupp for  E-mail exchanges in connection with
the Bass stable rank. I also thank the referee for his numerous linguistic comments
and for suggesting  to add a section on the topological stable rank of subalgebras of $A(\ov\D)$.

\end{document}